\documentclass[11pt]{article}
\usepackage{amsmath,amsthm}
\usepackage{amssymb}
\usepackage{enumerate}
\usepackage{graphicx,tikz}

\usepackage[mathscr]{eucal}
\usepackage[cm]{fullpage}
\usepackage{color}
\usepackage[english]{babel}
\usepackage[latin1]{inputenc}
\newcommand{\End}{\mathrm{End}}
\newcommand{\wEnd}{\mathrm{wEnd}}
\newcommand{\sEnd}{\mathrm{sEnd}}
\newcommand{\swEnd}{\mathrm{swEnd}}
\newcommand{\Aut}{\mathrm{Aut}}

\theoremstyle{plain}
\newtheorem{theorem}{Theorem}[section]
\newtheorem{proposition}[theorem]{Proposition}
\newtheorem{lemma}[theorem]{Lemma}
\newtheorem{corollary}[theorem]{Corollary}

\theoremstyle{remark}

\def\im{\mathop{\mathrm{Im}}\nolimits} 
\def\ker{\mathop{\mathrm{Ker}}\nolimits} 
 
\def\N{\mathbb N}

\numberwithin{equation}{section}
\def\inv{\mathop{\mathrm{inv}}\nolimits} 
\def\Inv{\mathop{\mathrm{Inv}}\nolimits} 
\def\rep{\mathop{\mathrm{rep}}\nolimits} 

\begin{document}

\title{Ranks of monoids of endomorphisms of a finite undirected path}

\author{I. Dimitrova\footnote{This work was developed within the 
FCT Project UID/MAT/00297/2013 of CMA.}, V.H. Fernandes\footnote{This work was developed within the 
FCT Project UID/MAT/00297/2013 of CMA and of Departamento de Matem\'atica da Faculdade de Ci\^encias e Tecnologia da Universidade Nova de Lisboa.}, J. Koppitz\footnote{This work was developed within the 
FCT Project UID/MAT/00297/2013 of CMA.}\, and T.M. Quinteiro\footnote{This work was developed within the 
FCT Project UID/MAT/00297/2013 of CMA and of Instituto Superior de Engenharia de Lisboa.}.
}

\maketitle

\renewcommand{\thefootnote}{}

\footnote{2010 \emph{Mathematics Subject Classification}: 05C38, 20M10, 20M20, 05C25}

\footnote{\emph{Keywords}: graph endomorphisms, paths, generators, rank.}

\renewcommand{\thefootnote}{\arabic{footnote}}
\setcounter{footnote}{0}

\begin{abstract}
In this paper we study the widely considered endomorphisms and weak endomorphisms of a finite undirected path from monoid generators perspective. Our main aim is to determine the ranks of the monoids $\wEnd P_n$ and $\End P_n$ of all weak endomorphisms and all endomorphisms of the undirected path $P_n$ with $n$ vertices. We also consider strong and strong weak endomorphisms of $P_n$.  
\end{abstract}

\section*{Introduction and Preliminaries} 

In the same way that automorphisms of graphs allow to establish natural connections between Graph Theory and Group Theory,  endomorphisms of graphs do the same between Graph Theory and Semigroup Theory. 
For this reason, it is not surprising that monoids of endomorphisms of graphs have been attracting the attention of several authors over the last decades. In fact, from combinatorial properties to more algebraic concepts have been extensively studied. 
Regularity, in the sense of Semigroup Theory, is one of the most studied notions. Regular semigroups constitutes a very important class in Semigroup Theory. A general solution to the problem, posed in 1988 by M\'arki \cite{Marki:1988}, of which graphs have a regular monoid of endomorphisms seems to be very difficult to obtain. Nevertheless, for some special classes of graphs, various authors studied and solved this question (for instance, see  \cite{Fan:1993,Fan:1997,Fan:2002,Hou&Gu:2016,Hou&Gu&Shang:2014,Hou&Gu&Shang:2015,Hou&Luo&Fan:2012,Hou&Song&Gu:2017,Knauer&Wanichsombat:2014,Li:2006,Li&Chen:2001,Lu&Wu:2013,Pipattanajinda&Knauer&Gyurov&Panma:2016,Wilkeit:1996}). 

\medskip 

In this paper, we focus our attention on a very important invariant of a semigroup or a monoid, which has been the subject of intensive research in Semigroup Theory. We are referring to the \textit{rank}, i.e. to the minimum number of generators of a semigroup or a monoid.

\medskip 

Let $\Omega$ be a finite set
with at least $2$ elements.
It is well-known that the full symmetric group  of $\Omega$
has rank $2$ (as a semigroup, monoid or group). Furthermore, the monoid of all transformations
and the monoid of all partial transformations of $\Omega$ have
ranks $3$ and $4$, respectively.
The survey \cite{Fernandes:2002} presents 
these results and similar ones for other classes of transformation monoids,
in particular, for monoids of order-preserving transformations and
for some of their extensions. More recently, for instance, the papers 
\cite{AlKharousi&Kehinde&Umar:2014s,Araujo&al:2015,Araujo&Schneider:2009,Cicalo&al:2015,Fernandes&al:2014,Fernandes&al:2016,Fernandes&al:2018ip,Fernandes&Quinteiro:2014,Fernandes&Sanwong:2014,Huisheng:2009,Zhao:2011,Zhao&Fernandes:2015} 
are dedicated to the computation of the ranks of certain (classes of transformation) semigroups or monoids.

\medskip 

Now, let $G=(V,E)$ be a simple graph (i.e. undirected, without loops and without multiple edges). 
Let $\alpha$ be a full transformation of $V$. We say that $\alpha$ is: 
\begin{itemize}
\item an \textit{endomorphism} of $G$ if $\{u,v\}\in E$ implies  $\{u\alpha,v\alpha\}\in E$, for all $u,v\in V$;
\item a \textit{weak endomorphism} of $G$ if $\{u,v\}\in E$ and $u\alpha\ne v\alpha$ imply  $\{u\alpha,v\alpha\}\in E$, for all $u,v\in V$;
\item a \textit{strong endomorphism} of $G$ if $\{u,v\}\in E$ if and only if  $\{u\alpha,v\alpha\}\in E$, for all $u,v\in V$;
\item a \textit{strong weak endomorphism} of $G$ if $\{u,v\}\in E$ and $u\alpha\ne v\alpha$ if and only if $\{u\alpha,v\alpha\}\in E$, for all $u,v\in V$;
\item an \textit{automorphism} of $G$ if $\alpha$ is a bijective strong endomorphism (i.e. $\alpha$ is bijective and $\alpha$ and $\alpha^{-1}$ are both endomorphisms). For finite graphs, any bijective endomorphism is an automorphism. 
\end{itemize}

Denote by:
\begin{itemize}
\item $\End G$ the set of all endomorphisms of $G$;
\item $\wEnd G$ the set of all weak endomorphisms of $G$;
\item $\sEnd G$ the set of all strong endomorphisms of $G$;
\item $\swEnd G$ the set of all strong weak endomorphisms of $G$;
\item $\Aut G$ the set of all automorphisms of $G$. 
\end{itemize}

Clearly, $\End G$, $\wEnd G$, $\sEnd G$, $\swEnd G$ and $\Aut G$ are monoids under composition of maps. Moreover, $\Aut G$ is also a group.  
It is also clear that $\Aut G\subseteq\sEnd G\subseteq\End G,    
[\text{respectively},  \swEnd G] \subseteq\wEnd G$ 
\begin{center}
\begin{picture}(55,140)(0,10)
\put(40,60){\line(-1,1){40}}
\put(40,60){\line(1,1){40}}
\put(80,100){\line(-1,1){40}}
\put(0,100){\line(1,1){40}}
\put(40,20){\line(0,1){40}}
\put(37,16.5){$\bullet$}\put(45,16.5){$\Aut G$}
\put(37,137.5){$\bullet$}\put(45,140){$\wEnd G$}
\put(77,98){$\bullet$}\put(84,100){$\End G$}
\put(-2,98){$\bullet$}\put(-45,100){$\swEnd G$}
\put(37.5,58){$\bullet$}\put(45,55){$\sEnd G$}
\end{picture}
\end{center}
(these inclusions may not be strict).

\smallskip 

Define $N(u)=\{v\in V\mid \{u,v\}\in E\}$ (the \textit{neighbors} of $u$), for all $u\in V$, 
and a binary relation $R_G$ on $V$ by $(u,v)\in R_G$ if and only if $N(u)=N(v)$, for all $u,v\in V$. 

\medskip 

Let $P_n$ be the undirected path with $n$ vertices. The number of endomorphisms of $P_n$ has been determined by Arworn 
\cite{Arworn:2009} (see also the paper \cite{Michels&Knauer:2009} by Michels and Knauer). 
In addition, several other combinatorial and algebraic properties of $P_n$ were also studied in these two papers and, 
for instance, in \cite{Arworn&Knauer&Leeratanavalee:2008,Hou&Luo&Cheng:2008}. 

\medskip 

In this paper, our main objective is to determine the ranks of the monoids $\wEnd P_n$ and $\End P_n$, 
what we will do in Section \ref{mainaim}, 
after presenting some more basic properties in Section \ref{basics}. The monoids $\swEnd P_n$, $\sEnd P_n$ and $\Aut P_n$ will also be considered in Section \ref{basics}. 

\medskip 

For general background on Semigroup
Theory and standard notation, we refer the reader to Howie's book \cite{Howie:1995}. 
On the other hand, regarding Algebraic Graph Theory, our main reference is Knauer's book \cite{Knauer:2011}.  

\section{Basic properties} \label{basics} 

Let $n\in\N$. The undirected \textit{path} with $n$ vertices may be defined as being the following simple graph:
$$
P_n=\left(\{1,\ldots,n\},\{\{i,i+1\}\mid i=1,\ldots,n-1\}\right). 
$$

For $u,v\in\{1,\ldots,n\}$, it is clear that $\{u,v\}$ is an edge of $P_n$ if and only if $|u-v|=1$. 
Hence, clearly: 

\begin{proposition}\label{form}
Let $\alpha$ be a full transformation of $\{1,\ldots,n\}$. Then, we have: 
\begin{enumerate}
\item $\alpha\in\wEnd P_n$ if and only if $|i\alpha-(i+1)\alpha|\le1$, for $i=1,\ldots,n-1$; 

\item $\alpha\in\End P_n$ if and only if $|i\alpha-(i+1)\alpha|=1$, for $i=1,\ldots,n-1$. 
\end{enumerate}
\end{proposition} 

\medskip

We also have the following property of $\wEnd P_n$. 

\begin{proposition}\label{interval} 
Let $\alpha\in\wEnd P_n$. If $I$ is an interval of $\{1,\ldots,n\}$ (with the usual order) 
then $I\alpha$ also is an interval of $\{1,\ldots,n\}$. In particular,  
$\im(\alpha)$ is an interval of $\{1,\ldots,n\}$. 
\end{proposition}
\begin{proof}
Let $\alpha\in\wEnd P_n$. Let $I=[u,v]$, with $1\le u\le v\le n$. 

If $|I\alpha|=1$ then, trivially, $I\alpha$ is an interval of $\{1,\ldots,n\}$. 

Now, suppose that $|I\alpha|\ge2$ and, in order to obtain a contradiction, 
admit that $I\alpha$ is not an interval of $\{1,\ldots,n\}$. 
Then, there exists $a\in I\alpha$ such that $a+1,\ldots,a+t-1\not\in I\alpha$ and $a+t\in I\alpha$, for some $t\ge2$. 
Take 
$$
\beta_1=\left(
\begin{array}{cccccccccc}
1&\cdots&u&u+1&\cdots&v-1&v&\cdots&n\\
u&\cdots&u&u+1&\cdots&v-1&v&\cdots&v
\end{array}
\right)
$$
and
$$
\beta_2=\left(
\begin{array}{cccccccccc}
1&\cdots&a&a+1&\cdots&a+t-1&a+t&\cdots&n\\
a&\cdots&a&a+1&\cdots&a+t-1&a+t&\cdots&a+t
\end{array}
\right). 
$$    
Clearly, $\beta_1,\beta_2\in\wEnd P_n$ and $\im(\beta_1\alpha\beta_2)=\{a,a+t\}$. Let $\gamma=\beta_1\alpha\beta_2$. 
Hence $\gamma\in\wEnd P_n$. 
Let $i=\min(a\gamma^{-1})$ and $k$ be the largest integer such that $\{i,\ldots,i+k\}\subseteq a\gamma^{-1}$. 
Then, as  $\im(\gamma)=\{a,a+t\}$, we have $i>1$ or $i+k<n$. 
If $i>1$ then $(i-1)\gamma=a+t$ and so $\{a,a+t\}=\{i\gamma,(i-1)\gamma\}$ is an edge of $P_n$, which is a contradiction. 
Then $i+k<n$, whence $(i+k+1)\gamma=a+t$ and so $\{a,a+t\}=\{(i+k)\gamma,(i+k+1)\gamma\}$ is an edge of $P_n$, which again is a contradiction. 
Thus, $I\alpha$ must be an interval of $\{1,\ldots,n\}$, as required. 
\end{proof} 

\medskip 

Next, we consider strong endomorphisms. 

\smallskip 

Let $\alpha\in\Aut P_n$. Let $i=1\alpha$. Then $2\alpha\in\{i-1,i+1\}$. Since $\alpha$ is a permutation of $\{1,\ldots,n\}$, 
if $2\alpha=i+1$ then $3\alpha=i+2, \ldots, (n-i+1)\alpha=n$ and so $n-i+1=n$ 
(otherwise $(n-i+2)\alpha=n-1=(n-i)\alpha$, which is a contradiction), i.e. $i=1$, 
whence  $\alpha={1\cdots n\choose 1\cdots n}$. 
On the other hand, 
if $2\alpha=i-1$ then $3\alpha=i-2, \ldots, i\alpha=1$ and so $i=n$ (otherwise $(i+1)\alpha=2=(i-1)\alpha$, which is a contradiction), 
whence  $\alpha={1\cdots n\choose n\cdots 1}$. 
Thus, we proved: 

\begin{theorem}\label{aut} 
$\Aut P_n=\{ {1\cdots n\choose 1\cdots n},{1\cdots n\choose n\cdots 1}\}$ {\em(}a cyclic group of order two, for $n\ge2${\em)}.
\end{theorem}

For $n\ne3$, it is easy to check that the relation $R_{P_n}$ is the identity.
Hence, as an immediate consequence of \cite[Proposition 1.7.15]{Knauer:2011}, we have:

\begin{theorem}
$\sEnd P_n=\Aut P_n$, for $n\ne3$.
\end{theorem}

Observe that  
$\sEnd P_3=\{{123\choose123},{123\choose321},{123\choose121},{123\choose212},{123\choose232},{123\choose323}\}$, 
whence $\Aut P_3\subsetneq\sEnd P_3$. 

\medskip 

For strong weak endomorphisms of $P_n$, we have: 

\begin{theorem}
$\swEnd P_n=\{ {1\cdots n\choose 1\cdots n},{1\cdots n\choose n\cdots 1},{1\cdots n\choose 1\cdots 1},\ldots,{1\cdots n\choose n\cdots n}\}$ 
{\em(}the automorphisms together with the constants{\em)}, for $n\ne3$.
\end{theorem}
\begin{proof}
The equality is obvious for $n\le 2$. Then, suppose that $n\ge4$ and 
let $\alpha\in\swEnd P_n\setminus\Aut P_n$. 

Let $i,j\in\{1,\ldots,n\}$ be such that $i<j$ and $i\alpha=j\alpha$ (notice that, as  $\alpha\in\swEnd P_n\setminus\Aut P_n$, 
such a pair of integers always exists). 
Let $k$ be the largest integer such that $0\le k\le i-1$ and $\{i-k,\ldots,i\}\alpha=\{i\alpha\}$. 
If $k<i-1$ then $i-k-1\ge1$ and $(i-k-1)\alpha=i\alpha\pm1=j\alpha\pm1$, whence $\{(i-k-1)\alpha,j\alpha\}$ is an edge of $P_n$ and so 
$\{i-k-1,j\}$ is also an edge of $P_n$, which is a contradiction since $i-k-1<j-1$. 
Hence, $k=i-1$ and so $\{1,\ldots,i\}\alpha=\{i\alpha\}$. A similar reasoning also allow us to deduce that $\{j,\ldots,n\}\alpha=\{i\alpha\}$. 

Now, take the largest integer $i$ such that there exist an integer $j$ such that $1\le i<j\le n$ and $i\alpha=j\alpha$. Then, we have 
$\{1,\ldots,i\}\alpha=\{i\alpha\}=\{j,\ldots,n\}\alpha$. 

If $i>1$ then we get $1\le i-1<i$ and $(i-1)\alpha=i\alpha$, whence $\{1,\ldots,i-1\}\alpha=\{i\alpha\}=\{i,\ldots,n\}\alpha$ 
and so $\alpha$ is a constant transformation. 

Thus, suppose that $i=1$. 
If $j>2$ then, given the choice of $i$, we must have $2\alpha=1\alpha\pm1=n\alpha\pm1$, whence $\{2\alpha,n\alpha\}$ is an edge of $P_n$ and so 
$\{2,n\}$ is also an edge of $P_n$, which is a contradiction since $n\ge4$. 
Therefore, $j=2$ and so $\alpha$ is also a constant transformation, as required. 
\end{proof} 

Notice that $\swEnd P_3=\{{123\choose123},{123\choose321},{123\choose121},{123\choose212},{123\choose232},{123\choose323},{123\choose111},{123\choose222},{123\choose333}\}$. 

\medskip 

It is a routine matter to show that $\swEnd P_3$ is generated by $\left\{{123\choose321},{123\choose212},{123\choose111}\right\}$, 
and $\swEnd P_n$ is generated by $\left\{{1\cdots n\choose n\cdots 1},{1\cdots n\choose 1\cdots 1},\ldots,
{1\;\;\cdots\;\;n\choose \lceil\frac{n}{2}\rceil\cdots \lceil\frac{n}{2}\rceil}\right\}$, 
for $n\ne3$. Furthermore, it is easy to deduce that these sets of generators have minimal cardinality. Hence, we obtain the following result.

\begin{theorem}
For $n\ge2$, $\swEnd P_n$ has rank $\lceil\frac{n}{2}\rceil+1$.
\end{theorem}

\subsection*{Regularity}

Recall that an element $s$ of a semigroup $S$ is called \textit{regular} if there exists $x\in S$ such that $s=sxs$. 
A semigroup is said to be \textit{regular} if all its elements are regular. 

Since $\Aut G$ is a group, for any graph $G$, then it is, trivially, a regular monoid. By the above properties and observations, 
it is clear that $\sEnd P_n$ and $\swEnd P_n$ are also regular monoids. Regarding $\End P_n$ and $\wEnd P_n$, we have:

\begin{proposition}\label{reg}
Let $\alpha\in\wEnd P_n$ {\em[}respectively, $\alpha\in\End P_n${\em]}. Then $\alpha$ is regular in $\wEnd P_n$ {\em[}respectively, in $\End P_n${\em]} if and only if there exists an interval $I$ of $\{1,\ldots,n\}$ {\em(}with the usual order{\em)} such that $I\alpha=\im(\alpha)$ 
and $|I|=|\im(\alpha)|$.  
\end{proposition} 
\begin{proof}
First, we suppose that $\alpha$ is regular in $\wEnd P_n$ [respectively, in $\End P_n$]. 
Then, there exists $\beta\in \wEnd P_n$ [respectively, $\beta\in\End P_n$] such that $\alpha=\alpha\beta\alpha$. 
Let $I=\im(\alpha\beta)$. Then, by Proposition \ref{interval}, $I$ is an interval of $\{1,\ldots,n\}$. 
Moreover, $I\alpha=(\im(\alpha\beta))\alpha=\im(\alpha\beta\alpha)=\im(\alpha)$. 
On the other hand, since $\alpha$ and $\alpha\beta$ are $\mathcal{R}$-related, then $\alpha$ and $\alpha\beta$ are $\mathcal{J}$-related, 
whence $|\im(\alpha)|=|\im(\alpha\beta)|$ and so $|I|=|\im(\alpha)|$. 

Conversely, admit that there exists an interval $I$ of $\{1,\ldots,n\}$ such that $I\alpha=\im(\alpha)$ 
and $|I|=|\im(\alpha)|$. 
If $|I|=1$ then $\alpha$ is a constant transformation (this case does not occur if $\alpha\in\End P_n$), 
whence $\alpha$ is an idempotent and so $\alpha$ is a regular element (of $\wEnd P_n$). Thus, suppose that $|I|\ge2$. 
Then $I=\{i,\ldots,j\}$, for some $1\le i< j\le n$, and the restriction of the transformation $\alpha$ to $I$ is injective. 
Hence, we have $i\alpha<\cdots<j\alpha$ or $i\alpha>\cdots>j\alpha$ (by a reasoning similar to the proof of Theorem \ref{aut}). 
Let $\beta$ be the transformation of $\{1,\ldots,n\}$ defined as follows. 
\begin{enumerate}
\item The restriction of $\beta$ to $\im(\alpha)$ is the inverse of the restriction of $\alpha$ to $I$, i.e.
$
\beta|_{\im(\alpha)} = \left(
\begin{array}{ccc}
i\alpha&\cdots&j\alpha\\
i& \cdots & j 
\end{array}
\right). 
$
\item Suppose $i\alpha<\cdots<j\alpha$ and let $\im(\alpha)^-=\{1,\ldots,i\alpha-1\}$ and $\im(\alpha)^+=\{j\alpha+1,\ldots,n\}$. 
\begin{enumerate}
\item If $i\alpha$ is odd (and $i\alpha\ge3$) then 
$
\beta|_{\im(\alpha)^-}=\left(
\begin{array}{ccccc}
1&2&\cdots&i\alpha-2&i\alpha-1\\
i&i+1&\cdots&i&i+1
\end{array}
\right). 
$
\item If $i\alpha$ is even then 
$
\beta|_{\im(\alpha)^-}=\left(
\begin{array}{ccccc}
1&2&\cdots&i\alpha-2&i\alpha-1\\
i+1&i&\cdots&i&i+1
\end{array}
\right). 
$
\item If $n-j\alpha$ is odd then 
$
\beta|_{\im(\alpha)^+}=\left(
\begin{array}{ccccc}
j\alpha+1&j\alpha+2&\cdots&n-1&n\\
j-1&j&\cdots&j&j-1
\end{array}
\right). 
$
\item If $n-j\alpha$ is even (and $n-j\alpha\ge2$) then
$
\beta|_{\im(\alpha)^+}=\left(
\begin{array}{ccccc}
j\alpha+1&j\alpha+2&\cdots&n-1&n\\
j-1&j&\cdots&j-1&j
\end{array}
\right). 
$
\end{enumerate}

\item Suppose $i\alpha>\cdots>j\alpha$ and let $\im(\alpha)^-=\{1,\ldots,j\alpha-1\}$ and $\im(\alpha)^+=\{i\alpha+1,\ldots,n\}$. 
\begin{enumerate}
\item If $j\alpha$ is odd (and $j\alpha\ge3$) then 
$
\beta|_{\im(\alpha)^-}=\left(
\begin{array}{ccccc}
1&2&\cdots&j\alpha-2&j\alpha-1\\
j&j-1&\cdots&j&j-1
\end{array}
\right). 
$
\item If $j\alpha$ is even then 
$
\beta|_{\im(\alpha)^-}=\left(
\begin{array}{ccccc}
1&2&\cdots&j\alpha-2&j\alpha-1\\
j-1&j&\cdots&j&j-1
\end{array}
\right). 
$
\item If $n-i\alpha$ is odd then 
$
\beta|_{\im(\alpha)^+}=\left(
\begin{array}{ccccc}
i\alpha+1&i\alpha+2&\cdots&n-1&n\\
i+1&i&\cdots&i&i+1
\end{array}
\right). 
$
\item If $n-i\alpha$ is even (and $n-i\alpha\ge2$) then
$
\beta|_{\im(\alpha)^+}=\left(
\begin{array}{ccccc}
i\alpha+1&i\alpha+2&\cdots&n-1&n\\
i+1&i&\cdots&i+1&i
\end{array}
\right). 
$
\end{enumerate}
\end{enumerate} 
It is clear that $\beta\in \End P_n$ (and so $\beta\in\wEnd P_n$) and $\alpha=\alpha\beta\alpha$. 
Hence, $\alpha$ is regular, as required. 
\end{proof}

It is a routine matter to check that $\End P_n$ is regular for $n\le5$, and that $\wEnd P_n$ is regular for $n\le3$. 
On the other hand, by Proposition \ref{reg}, it is clear that 
$$
\left(
\begin{array}{cccccccc}
1&2&3&4&5&6&\cdots&n\\
1&2&3&2&3&4&\cdots&n-2
\end{array}
\right)
$$
is not a regular element of $\End P_n$ for $n\ge6$,  and 
$$
\left(
\begin{array}{cccccc}
1&2&3&4&\cdots&n\\
1&2&2&3&\cdots&n-1
\end{array}
\right)
$$
is not a regular element of $\wEnd P_n$ for $n\ge4$. Thus, we have: 

\begin{corollary}
The monoid $\wEnd P_n$ {\em[}respectively, $\End P_n${\em]} is regular if and only if $n\le3$ {\em[}respectively, $n\le5${\em]}.
\end{corollary}

\subsection*{Cardinality}

It is clear that $|\Aut P_1|=|\sEnd P_1|=|\swEnd P_1|=|\End P_1|=|\wEnd P_1|=1$; 
$|\Aut P_n|=|\sEnd P_n|=2$, for $n=2$ and $n\ge4$; 
$|\swEnd P_n|=n+2$, for $n=2$ and $n\ge4$; 
$|\Aut P_3|=2$, $|\sEnd P_3|=6$ and $|\swEnd P_3|=9$. 

A formula for $|\End P_n|$ was given by Arworn \cite{Arworn:2009} in 2009. 
Regarding $|\wEnd P_n|$, we give a formula below. 

First, we recursively define a family $a(r,i)$, with $1\le r\le n-2$ and $1\le i\le n-1$, of integers:  
\begin{itemize}
\item $a(1,1)=a(1,2)=1$;
\item $a(1,p)=0$, for $3\le p\leq n-1$; 
\item For $2\leq k\leq n-2$, 
\begin{description}
\item $a(k,1)=a(k-1,1)+a(k-1,2)$, and 
\item $a(k,p)=a(k-1,p-1)+a(k-1,p)+a(k-1,p+1)$, for $2\leq p\leq n-2$; 
\end{description}
\item $a(k,n-1)=0$, for $2\le k\le n-3$; 
\item $a(n-2,n-1)=1$. 
\end{itemize}
Next,  let  
$
b(r)=2\sum_{i=1}^{n-1}a(r,i), 
$
for $1\leq r\leq n-2$.

Then, we have: 

\begin{theorem}\label{card}
$\displaystyle |\wEnd P_n| = 3^{n-2}(3n-2)-\sum_{r=1}^{n-2}3^{n-r-2}b(r)$. 
\end{theorem}

\begin{proof}
Let $\alpha \in \wEnd P_{n}$. We will mainly use the fact that 
$(i+1)\alpha \in\{i\alpha ,i\alpha +1\}$ if $i\alpha =1$, 
$(i+1)\alpha \in \{i\alpha -1,i\alpha,i\alpha +1\}$ if $i\alpha \in \{2,\ldots ,n-1\}$, 
and $(i+1)\alpha \in\{i\alpha ,i\alpha -1\}$ if $i\alpha =n$, for $i\in \{1,\ldots ,n-1\}$. 
For $i\in \{1,\ldots ,n-1\}$, let $c(i)$ be the number of possibilities for $
\alpha |_{\{1,\ldots ,i+1\}}$. Observe that $c(n-1)=|\wEnd P_n|$. 

First, we calculate $c(1)$. If $1\alpha \in \{2,\ldots ,n-1\}$ then we have
three possibilities for $2\alpha $. On the other hand, 
if $1\alpha \in \{1,n\}$ then we only have two possibilities for $2\alpha$. 
This shows that $c(1)=3(n-2)+2\cdot 2=3n-2$. 
Now, let $i\in \{1,\ldots ,n-2\}$ and suppose that $c(i)$ is known.
We will show that $c(i+1)=3c(i)-b(i)$, where $b(i)$ denotes the number of
possibilities for $(i+1)\alpha \in \{1,n\}$. In fact, 
if $(i+1)\alpha \in\{2,\ldots ,n-1\}$ then we have three possibilities for $(i+2)\alpha$ and, 
on the other hand, 
if $(i+1)\alpha \in \{1,n\}$ then we have only two possibilities for $(i+2)\alpha$. 
This shows that $c(i+1)=3(c(i)-b(i))+2b(i)=3c(i)-b(i)$. 
In this setting, we deduce that $c(n-1)=3^{n-2}(3n-2)-\sum_{r=1}^{n-2}3^{n-r-2}b(r)$, by 
performing successive replacements. 

It remains to calculate $b(r)$, for $r\in \{1,\ldots ,n-2\}$. 

For $k,p\in\{1,\ldots ,n-1\}$, let $a(k,p)$ denote the number of possibilities for 
$(k+1)\alpha =1$ and $1\alpha =p$. We will prove that $a(k,p)$ can be defined
as above. Clearly, $a(1,1)=a(1,2)=1$ and $a(1,p)=0$, for 
$p\in \{3,\ldots ,n-1\}$. 
Now, let $k\in \{2,\ldots ,n-2\}$ and suppose that 
$a(k-1,p)$ is known, for $p\in \{2,\ldots ,n-2\}$. 
If $1\alpha =1$ then $2\alpha \in \{1,2\}$ and $a(k,1)$ is the number of all possibilities for 
$k\alpha =1$, whenever $1\alpha =$ $1$ or $1\alpha =2$, i.e. 
$a(k,1)=a(k-1,1)+a(k-1,2)$. If $1\alpha \in \{2,\ldots ,n-2\}$ then 
$2\alpha\in \{1\alpha -1,1\alpha ,1\alpha +1\}$ and $a(k,p)$ is the number of
possibilities that $k\alpha =1$, whenever $1\alpha =p-1$ or $1\alpha =p$ or 
$1\alpha =p+1$, i.e. $a(k,p)=a(k-1,p-1)+a(k-1,p)+a(k-1,p+1)$. 
Clearly, $a(k,n-1)=0$, whenever $k<n-2$. 
Notice that $1\alpha =n-1$ and $(n-1)\alpha =1$
implies $r\alpha =n-1-r+1$, for $1\leq r\leq n-1$. Hence, there is only one
possibility for $1\alpha =n-1$ and $(n-1)\alpha =1$, i.e. $a(n-2,n-1)=1$.
Moreover, it is clear that  $k\alpha \neq 1$, whenever $1\alpha =n$ and $k<n$. 

Hence, for $a(r,i)$ as defined above, we have that 
$\sum_{i=1}^{n-1}a(r,i)$ is the number of possibilities
for $(r+1)\alpha =1$. 
Dually, $\sum_{i=1}^{n-1}a(r,i)$
is also the number of possibilities for $(r+1)\alpha =n$. 
Therefore, $b(r)=2\sum_{i=1}^{n-1}a(r,i)$, as required. 
\end{proof}

The table below gives us an idea of the size of $\wEnd P_n$. 

\begin{center}
$\begin{array}{|c|c|}\cline{1-2}
n & |\wEnd P_n|  \\ \cline{1-2}  
1 & 1\\ \cline{1-2}
2 & 4\\ \cline{1-2}
3 & 17\\ \cline{1-2}
4 &  68    \\ \cline{1-2}
5  &   259   \\ \cline{1-2}
6  &   950  \\ \cline{1-2}
7  &   3387  \\ \cline{1-2}
8  &   11814  \\ \cline{1-2}
\end{array}$ \quad
$\begin{array}{|c|c|c|}\cline{1-2}
n &  |\wEnd P_n| \\\cline{1-2}
9  &   40503   \\ \cline{1-2}
10  &  136946  \\ \cline{1-2}
11 &   457795\\ \cline{1-2}
12 &    1515926 \\ \cline{1-2}
13 &   4979777  \\ \cline{1-2}
14 &   16246924 \\ \cline{1-2}
15 &   52694573   \\ \cline{1-2}
16 &   170028792  \\ \cline{1-2}
\end{array}$
\end{center}

The formula given by Theorem \ref{card} allows us to calculate the cardinal of $\wEnd P_n$, even for larger $n$. 
For instance, we have $|\wEnd P_{100}|=15116889835751504709361077940682197429012095346416$.

\section{The ranks of the monoids $\End P_n$ and $\wEnd P_n$} \label{mainaim} 

Let
$$
\tau=\left(
\begin{array}{ccccc}
1&2&\cdots&n-1&n\\
n&n-1&\cdots&2&1
\end{array}
\right),
$$
$$
\alpha_i=\left(
\begin{array}{ccccccccccc}
1&2&\cdots&i-1&i&i+1&i+2&\cdots&n-1&n\\
i+1&i&\cdots&3&2&1&2&\cdots&n-i-1&n-i
\end{array}
\right),
$$
for $i=1,\ldots,n-2$, and 
$$
\beta_{j,i}=\left(
\begin{array}{ccccccccccccccc}
1&2&\cdots&i&i+1&\cdots&i+j&i+j+1&i+j+2&\cdots&i+2j&i+2j+1&i+2j+2&\cdots&n\\
1&2&\cdots&i&i+1&\cdots&i+j&i+j+1&i+j&\cdots&i+2&i+1&i+2&\cdots&n-2j
\end{array}
\right),
$$
for $j=1,\ldots,\lfloor\frac{n-3}{3}\rfloor$ and
$i=1,\ldots,n-3j-2$.
Let
$$
A'=\{\tau\}\cup\left\{\alpha_i\mid i=1,\ldots,n-2\right\}\cup\left\{\beta_{j,i}\mid j=1,\ldots,\lfloor\frac{n-3}{3}\rfloor,\, i=1,\ldots,n-3j-2\right\}.
$$
Also, let 
$$
A''=\{\tau\}\cup\left\{\alpha_i\mid i=1,\ldots,n-2\right\}.
$$

\begin{lemma}\label{le2}
Let $\alpha \in \End P_n$. Then $\{\alpha' \in \End P_n \mid \ker(\alpha') = \ker(\alpha)\} \subseteq \langle A'', \alpha \rangle$.
\end{lemma}
\begin{proof}
Let us suppose that 
$$
\alpha=\left(
\begin{array}{cccc}
X_1 &X_2&\cdots &X_k\\
i_1&i_2&\cdots&i_k
\end{array}
\right),
$$
with $\im(\alpha)=\{i_1 <i_2 <\cdots <i_k\}$ and $X_t=i_t\alpha^{-1}$, for $t=1,\ldots,k$, for some $1<k\le n$. 
Since $\im(\alpha)$ is an interval of $\{1,\ldots,n\}$, by Proposition \ref{interval}, then we have $i_t=i_1+t-1$, for $t=2,\ldots,k$.  

Since $x\alpha_i = x-i$, for all $x \in \{i+1, \ldots, n\}$, we obtain
$$
\alpha\alpha_i=\left(
\begin{array}{cccc}
X_1&X_2&\cdots&X_k\\
i_1-i &i_2-i&\cdots &i_k-i
\end{array}
\right) \in \langle A'', \alpha \rangle,
$$
with $\im(\alpha\alpha_i)=\{i_1-i <i_2-i <\cdots <i_k-i\}$, for $i = 1,2,\ldots,i_1-1$. 
Next, consider the transformation
$$
\tau\alpha_i\tau=\left(
\begin{array}{ccccccccc}
1&2&\cdots&n-i-1&n-i&n-i+1&\cdots&n\\
1+i&2+i&\cdots&n-1&n&n-1&\cdots&n-i
\end{array}
\right) \in \langle A'' \rangle, 
$$
for $i = 1,2,\ldots,n-2$. As $x(\tau\alpha_i\tau) = x+i$, for all $x \in \{1, \ldots, n-i\}$, we also get 
$$
\alpha\tau\alpha_i\tau=\left(
\begin{array}{cccc}
X_1&X_2&\cdots&X_k\\
i_1+i &i_2+i&\cdots &i_k+i
\end{array}
\right) \in \langle A'', \alpha \rangle,
$$
with $\im(\alpha\tau\alpha_i\tau)=\{i_1+i <i_2+i <\cdots <i_k+i\}$, for $ i = 1,2,\ldots,n-i_k$. 

\smallskip 

Thus, so far we proved that $\alpha'\in \langle A'', \alpha \rangle$, for all 
$$
\alpha'=\left(
\begin{array}{cccc}
X_1 &X_2&\cdots &X_k\\
j_1&j_2&\cdots&j_k
\end{array}
\right) \in\End P_n
$$
such that $\im(\alpha)=\{j_1 <j_2 <\cdots <j_k\}$ (and $X_t=j_t\alpha^{-1}$, for $t=1,\ldots,k$). 

\smallskip 

Now, take  
$$
\alpha'=\left(
\begin{array}{cccc}
X_1 &X_2&\cdots &X_k\\
j_1&j_2&\cdots&j_k
\end{array}
\right) \in\End P_n
$$
such that $\ker(\alpha')=\ker(\alpha)$ (with $X_t=j_t\alpha^{-1}$, for $t=1,\ldots,k$). 

Suppose there exists $p\in \{1,\ldots ,k-1\}$ such that $| j_{p}-j_{p+1} | >1$. 
Let $X^{-}=\bigcup \left\{ X_{j}\mid j\leq p\right\}$. 
Then, as $p<k$, we have $X^- \subsetneq\{1,\ldots ,n\}$ and so there exists 
$x\in X^{-}$ such that $x+1\in \{1,\ldots ,n\}\setminus X^{-}$ or $x-1\in\{1,\ldots ,n\}\setminus X^{-}$. 
Let us admit, without loss of generality, that $x+1\in \{1,\ldots ,n\}\setminus X^{-}$. 
Let $j\in \{1,\ldots ,p\}$ be such that $x\in X_{j}$.
Since $\alpha \in \End P_{n}$ is such that $X_t=i_t\alpha^{-1}$ 
and $i_t=i_1+t-1$, for $t=1,\ldots,k$, 
we can conclude that $x+1\in X_{j-1}\cup X_{j+1}$ (with $X_{j-1}=\emptyset$, if $j=1$).
As $x+1\not\in X^{-}$, it follows that $x+1\in X_{j+1}$ and $j+1>p$. 
Therefore $j=p$ and so, by Proposition \ref{form}, 
we have $1< | j_{p}-j_{p+1} | = | x\alpha'-(x+1)\alpha' | =1$, which is a contradiction. 

Thus, $| j_{t}-j_{t+1}| =1$, for all $t\in \{1,\ldots ,k-1\}$. This provides $j_{1}<j_{2}<\cdots <j_{k}$ or $j_{k}<j_{k-1}<\cdots <j_{1}$. 
If $j_{1}<j_{2}<\cdots <j_{k}$ then, as proved above, we have $\alpha'\in \left\langle A'',\alpha \right\rangle$. 
On the other hand, suppose that $j_{k}<j_{k-1}<\cdots <j_{1}$ and consider 
$$
\alpha'\tau=\left(
\begin{array}{ccccccc}
X_1& X_2 &\cdots &X_k\\
j_1\tau&j_2\tau&\cdots&j_k\tau
\end{array}
\right) \in \End P_n.   
$$
Then, as $\tau$ is a permutation of $\{1,\ldots,n\}$, we obtain $\ker(\alpha'\tau)=\ker(\alpha')=\ker(\alpha)$. 
Furthermore, we also have $\im(\alpha'\tau)=\{j_1\tau <j_2\tau <\cdots <j_k\tau\}$. Hence, 
again as proved above, we have $\alpha'\tau\in \langle A'', \alpha \rangle$. 
Since $\alpha'=(\alpha'\tau)\tau$, it follows that $\alpha'\in \langle A'', \alpha \rangle$. 

Thus, we proved that 
$\{\alpha' \in \End P_n \mid \ker(\alpha') = \ker(\alpha)\} \subseteq \langle A'', \alpha \rangle$, as required. 
\end{proof}

\smallskip 

Let $\alpha\in\wEnd P_n$. We say that $i\in\{2,\ldots,n-1\}$ is an \textit{inversion} of $\alpha$ if $(i-1)\alpha=(i+1)\alpha\ne i\alpha$. 
Denote by 
$
\Inv(\alpha)
$
the set of all inversions of $\alpha$ and by $\inv(\alpha)$ the number of elements of $\Inv(\alpha)$. 
Notice that, if $\alpha\in\End P_n$ then $i\in\{2,\ldots,n-1\}$ is an \textit{inversion} of $\alpha$ if and only if $(i-1)\alpha=(i+1)\alpha$. 

\smallskip 

For elements of $\End P_n$, we have: 

\begin{lemma}\label{invker}
Let $\alpha,\beta\in \End P_n$. Then $\ker(\alpha)=\ker(\beta)$ if and only if $\Inv(\alpha)=\Inv(\beta)$.
\end{lemma}

\begin{proof}
Clearly, $\ker(\alpha) = \ker(\beta) $ implies $\Inv(\alpha)=\Inv(\beta)$. 
Conversely, admit that $\Inv(\alpha)=\Inv(\beta)$. Let $\Inv(\alpha)=\{i_{1}<\cdots <i_{k}\}$, for some $k\in \{1,\ldots ,n-2\}$, and 
define $i_{0}=1$ and $i_{k+1}=n$. 
For any $x\in \{1,\ldots,n\}$, let  $p(x)\in \{0,\ldots,k\}$ and $r(x)\in \{0,\ldots,i_{p(x)+1}-i_{p(x)}-1\}$ 
be such that $x=i_{p(x)}+r(x)$. 
Since $\left\vert x\alpha -(x+1)\alpha \right\vert =\left\vert x\beta -(x+1)\beta\right\vert =1$,  
for all $x\in \{1,\ldots,n-1\}$, and $\Inv(\alpha)=\Inv(\beta)$, 
there exist $a_{\alpha},a_{\beta}\in \{1,2\}$ such that
$$
x\alpha = 1\alpha +\underset{j=0}{\overset{p(x)-1}{\sum}}(-1)^{j+a_{\alpha}}(i_{j+1}-i_{j})+(-1)^{p(x)+a_{\alpha}}r(x)
$$
and
$$
x\beta = 1\beta +\underset{j=0}{\overset{p(x)-1}{\sum}}(-1)^{j+a_{\beta}}(i_{j+1}-i_{j})+(-1)^{p(x)+a_{\beta }}r(x).
$$
Therefore, $(x,y)\in \ker(\alpha)$ if and only if $x\alpha =y\alpha$
if and only if 
$$
1\alpha +\underset{j=0}{\overset{p(x)-1}{\sum}}(-1)^{j+a_{\alpha }}(i_{j+1}-i_{j})+(-1)^{p(x)+a_{\alpha}}r(x) = 
1\alpha + \underset{j=0}{\overset{p(y)-1}{\sum}}(-1)^{j+a_{\alpha}}(i_{j+1}-i_{j})+(-1)^{p(y)+a_{\alpha}}r(y)
$$ 
if and only if (by multiplication with $(-1)^{\left\vert a_{\alpha}-a_{\beta}\right\vert}$
and addition of $1\beta - (1\alpha)(-1)^{\left\vert a_{\alpha}-a_{\beta}\right\vert}$) 
$$
1\beta + \underset{j=0}{\overset{p(x)-1}{\sum}}(-1)^{j+a_{\beta}}(i_{j+1}-i_{j})+(-1)^{p(x)+a_{\beta}}r(x) =
1\beta + \underset{j=0}{\overset{p(y)-1}{\sum}}(-1)^{j+a_{\beta}}(i_{j+1}-i_{j})+(-1)^{p(y)+a_{\beta}}r(y)
$$
if and only if $x\beta =y\beta$ if and only if $(x,y)\in \ker(\beta)$. 
Thus, $\ker(\alpha) = \ker(\beta)$, as required. 
\end{proof}

The following lemma is an immediate consequence of Lemmas \ref{le2} and \ref{invker}. 

\begin{lemma}\label{rem}
Let $\alpha, \beta \in \End P_{n}$ be such that $\Inv(\alpha)=\Inv(\beta)$. Then $\alpha\in \langle A''\rangle$ if and only if $\beta\in \langle A''\rangle$. 
Moreover, $\alpha\in \langle A'\rangle$ if and only if $\beta\in \langle A'\rangle$. 
\end{lemma}

Now, we can prove: 

\begin{lemma}\label{ger}
$\End P_n = \langle A' \rangle$.
\end{lemma}
\begin{proof}
Let $\alpha \in \End P_n$. We will proceed by induction on $\inv(\alpha)$.

If $\inv(\alpha) = 0$ then $\alpha = \tau \in A'$ or $\alpha = 1 = \tau^2 \in \langle A' \rangle$.

If $\inv(\alpha) = 1$ then $\Inv(\alpha)=\Inv(\alpha_{i})$, for some $i\in \{1,2,\ldots,n-2\}$, and so
$\alpha \in \left\langle A' \right\rangle$, by Lemma \ref{rem}.

Now, let $r \geq 1$ and suppose, by induction hypothesis, that $\alpha \in \langle A' \rangle$, 
for all $\alpha \in \End P_n$ with $\inv(\alpha) \leq r$.

Let $\alpha \in \End P_n$ be such that $\inv(\alpha) = r+1$. By (the proof of) Lemma \ref{le2}, we can assume, without loss of generality, that $1 \in \im (\alpha)$.
Let $\Inv(\alpha) = \{i_1 < i_2 < \cdots < i_{r+1}\}$ and define $i_0 = 1$ and $i_{r+2} = n$. Let $b = \max \im (\alpha)$. 
Notice that $b \geq 2$ and, as $\im(\alpha) = \{1,\ldots, b\}$ (by Proposition \ref{interval}) and $\inv(\alpha) = r + 1 \geq 2$, 
we get $b \leq n-2$. Clearly, we have $1\alpha^{-1}, b\alpha^{-1} \subseteq \{i_\ell \mid \ell \in \{0,1,\ldots,r+2\}\}$.

We will consider three cases: (1) there exists $k \in \{1,\ldots,r+1\}$ such that $i_k\alpha = 1$; (2) there exists $k \in \{1,\ldots,r+1\}$ such that  $i_k\alpha = b$; (3) $1\alpha = 1$ and $n\alpha = b$ (or $1\alpha = b$ and $n\alpha = 1$) and $x\alpha \not\in\{1,b\}$, 
for all $x \in \{2,\ldots,n-1\}$.

\smallskip 

{\sc case} 1. Suppose that there exists $k \in \{1,\ldots,r+1\}$ such that $i_k\alpha = 1$.
Let $a = \max\{x\alpha \mid i_k \leq x \leq n\}$ and
define a transformation $\beta$ of $\{1,\ldots,n\}$ by
$$
x\beta = \left\{
\begin{array}{ll}
x\alpha +a-1,& x < i_k\\
a+1-x\alpha, & x \geq i_k.
\end{array}
\right.
$$

Then $\beta \in \End P_n$. In fact, since $\alpha \in \End P_n$, we have $|x\beta - (x+1)\beta| = 1$, 
for all $x \in \{1,\ldots,n-1\}\setminus \{i_k-1\}$.
Moreover, from $(i_k-1)\beta = (i_k-1)\alpha+a-1 = i_k\alpha \pm 1+a-1 = 1 \pm 1+a-1 = a\pm 1$ and $i_k\beta = a+1-i_k\alpha = a+1-1 = a$
(since $i_k\alpha = 1$), it follows that $|(i_k-1)\beta - i_k\beta| = 1$.

Next, we show that $\Inv(\beta) = \Inv(\alpha)\setminus \{i_{k}\}$. Clearly,
if $x\in\{1,\ldots, i_{k}-2\} \cup \{i_{k}+1,\ldots,n\}$ then $x\in \Inv(\beta)$ if and only if $x\in \Inv(\alpha)$. Also
$$
\begin{array}{rcl}
    i_{k}-1\in \Inv(\beta) & \Leftrightarrow& (i_{k}-2)\beta=(i_{k})\beta \\
     & \Leftrightarrow & (i_{k}-2)\alpha+a-1=a+1-i_{k}\alpha \\
     & \Leftrightarrow & (i_{k}-2)\alpha=1=i_{k}\alpha \\
     & \Leftrightarrow & i_{k}-1\in \Inv(\alpha).
  \end{array}
  $$
On the other hand, since $i_k\alpha = 1$, we have $(i_k-1)\alpha = (i_k+1)\alpha = 2$. Then
$$
(i_k-1)\beta = (i_k-1)\alpha +a-1 = 2+a-1 = a+1,
$$
$$
i_k\beta = a+1-i_k\alpha = a+1-1 = a
$$
and
$$
(i_k+1)\beta = a+1-(i_k+1)\alpha = a+1-2 = a-1,
$$
whence $i_k \notin \Inv(\beta)$.
Therefore, $\inv(\beta) = r$ and so, by induction, we have $\beta \in \langle A' \rangle$.

Finally, within this case, we prove that $\Inv(\alpha) = \Inv(\beta\alpha_{a-1})$.
Let $i \in \Inv(\alpha)\cap \{1,\ldots,i_k-2\}$. Then
$$
\begin{array}{rcl}
    (i-1)\beta\alpha_{a-1}=(i+1)\beta\alpha_{a-1} & \Leftrightarrow& ((i-1)\alpha+a-1)\alpha_{a-1}=((i+1)\alpha+a-1)\alpha_{a-1} \\
    & \Leftrightarrow & (i-1)\alpha=(i+1)\alpha
  \end{array}$$
(by the definition of $\alpha_{a-1}$ restricted to $\{a,\ldots,n\}$).
Let $i \in \Inv(\alpha)\cap \{i_k+1,\ldots,n\}$. Then
$$
\begin{array}{rcl}
    (i-1)\beta\alpha_{a-1}=(i+1)\beta\alpha_{a-1} & \Leftrightarrow & (a+1-(i-1)\alpha)\alpha_{a-1}=(a+1-(i+1)\alpha)\alpha_{a-1} \\
     & \Leftrightarrow & (i-1)\alpha=(i+1)\alpha
  \end{array}
$$
(by the definition of $\alpha_{a-1}$ restricted to $\{1,\ldots,a\}$).
Moreover, we have
$$
(i_{k}-2)\beta\alpha_{a-1}=(i_{k})\beta\alpha_{a-1}\Leftrightarrow ((i_{k}-2)\alpha+a-1)\alpha_{a-1}=(a+1-(i_{k})\alpha)\alpha_{a-1}.
$$
Thereby, if $i_{k}-1 \in \Inv(\alpha)$ then $(i_{k}-2)\alpha=i_{k}\alpha=1$, 
whence $(i_{k}-2)\alpha+a-1= a = a+1-i_{k}\alpha$ and so 
$(i_{k}-2)\beta\alpha_{a-1}=i_{k}\beta\alpha_{a-1}$, i.e. 
$i_{k}-1 \in \Inv(\beta\alpha_{a-1})$. 
Conversely, if $i_{k}-1 \notin \Inv(\alpha)$ then $(i_{k}-2)\alpha \ne i_{k}\alpha$ and, as $(i_{k}-1)\alpha=2$, we have $(i_{k}-2)\alpha=3$, from which follows $(i_{k}-2)\beta\alpha_{a-1}=(a+2)\alpha_{a-1}=3\ne 1=i_{k}\beta\alpha_{a-1}$ and so $i_{k}-1\notin \Inv(\beta\alpha_{a-1})$.
It remains to show that $i_{k}\in \Inv(\beta\alpha_{a-1})$. 
In fact, since $a \in \Inv(\alpha_{a-1})$ and $(i_{k}-1)\alpha=2=(i_{k}+1)\alpha$, 
we obtain 
$$
(i_{k}-1)\beta\alpha_{a-1}=((i_{k}-1)\alpha+a-1)\alpha_{a-1}=(a+1)\alpha_{a-1}=
(a-1)\alpha_{a-1}=(a+1-(i_{k}+1)\alpha)\alpha_{a-1}=(i_{k}+1)\beta\alpha_{a-1}.
$$
Thus, we have $\Inv(\alpha) = \Inv(\beta\alpha_{a-1})$. 
Since $\beta \in \langle A' \rangle$, then $\beta\alpha_{a-1} \in \langle A' \rangle$ and so, 
by Lemma \ref{rem}, we have $\alpha \in \langle A' \rangle$.

\smallskip 

{\sc case} 2. Suppose now that there exists $k \in \{1,\ldots,r+1\}$ such that $i_k\alpha = b$. 
Recall that $\im(\alpha) = \{1,\ldots, b\}$.
Consider the transformation $\alpha\tau\alpha_{n-b} \in \End P_n$. 
Since $\tau$ is a permutation of $\{1,\ldots,n\}$ and $\alpha_{n-b}$ is injective
in $\{n-b+1,\ldots,n\} = \im(\alpha)\tau = \im(\alpha\tau)$, then $\ker(\alpha\tau\alpha_{n-b}) = \ker(\alpha)$, 
i.e. $\Inv(\alpha\tau\alpha_{n-b}) = \Inv(\alpha)$, by Lemma \ref{invker}. 
Hence, $\alpha \in \langle A' \rangle$ if and only if $\alpha\tau\alpha_{n-b} \in \langle A' \rangle$, by Lemma \ref{rem}. 
Observe also that $i_k \in b\alpha^{-1} = 1(\alpha\tau\alpha_{n-b})^{-1}$ and so, in particular, 
we also have $i_k(\alpha\tau\alpha_{n-b}) = 1$, i.e. $\alpha\tau\alpha_{n-b}$ satisfies the condition of {\sc case}  1. 
Therefore, $\alpha\tau\alpha_{n-b} \in \langle A' \rangle$ and so we have $\alpha \in \langle A' \rangle$.

\smallskip 

{\sc case}  3. Next, we suppose that $\{1,b\}\alpha^{-1} = \{1,n\}$. 
Without loss of generality, let $1\alpha=1$ and $n\alpha=b$ (if $1\alpha=b$ and $n\alpha=1$ then $1\tau\alpha=1$, $n\tau\alpha=b$ and, by Lemma \ref{rem}, $\alpha \in \langle A' \rangle$ if and only if $\tau\alpha \in \langle A' \rangle$).

First, let us admit that $r=1$, i.e. $\Inv(\alpha)=\{i_{1},i_{2}\}$. Let $j=i_{2}-i_{1}$ and $i=2i_{1}-i_{2}-1$. 
As $1\alpha=1$ and $n\alpha=b$, then $i_{1}\alpha=i_{1}$ and $i_{2}\alpha=i_{1}-(i_{2}-i_{1})=2i_{1}-i_{2}=i+1$.  
In addition, from $b-i_{2}\alpha=n-i_{2}$, we obtain 
$$
b=n+i_{2}\alpha-i_{2}=n+(2i_{1}-i_{2})-i_{2}=n+2i_{1}-2i_{2}=n-2j.
$$
As $2 \leq i_2\alpha < i_1\alpha \leq b-1$, we have $b=n-2j \geq i_{1}\alpha+1 = i_{1}+1$, whence $i_{1}\le n-2j-1$, 
and $i_{1}\alpha-i_{2}\alpha\leq b-3$.
Thus,  
$$
i=2i_{1}-i_{2}-1\leq n-2j-1 + (i_{1}-i_{2})-1=n-3j-2
$$
and
$$
j=i_{2}-i_{1}=i_{1}\alpha-i_{2}\alpha\leq b-3=n-2j-3 \quad\Rightarrow\quad 3j\leq n-3 \quad\Rightarrow\quad j\leq \lfloor\frac{n-3}{3}\rfloor.$$
Therefore, we may consider $\beta_{j,i}\in A'$ and, clearly, we have $\alpha=\beta_{j,i}$. 
Hence $\alpha\in\langle A' \rangle$.

Now, suppose that $r > 1$. 
Define $c=\max\{i_1\alpha, \ldots, i_r\alpha\}$ and let $k \in \{1,2,\ldots,r\}$ be such that $i_k\alpha = c$. 
Since $i_{r+1}\alpha<b=n\alpha$, we have $i_{r+1}\alpha<i_r\alpha\le c$. 
Also, define $d=\min\{i_{k+1}\alpha, \ldots, i_{r+1}\alpha\}$ and let $\ell \in \{k+1,\ldots,r+1\}$ be such that $i_\ell\alpha = d$. 
Furthermore, we define a transformation $\gamma$ of $\{1,\ldots,n\}$ by
$$
x\gamma = \left\{
\begin{array}{ll}
x\alpha, & x < i_k\\
2c-x\alpha, & i_k \leq x \leq i_\ell\\
x\alpha + 2c-2d, & x > i_\ell.
\end{array}
\right.
$$
Then $\gamma \in \End P_n$. In fact, since $\alpha \in \End P_n$, we have $|x\gamma - (x+1)\gamma| = 1$, 
for all $x \in \{1,\ldots,n-1\}\setminus \{i_k-1, i_\ell\}$.
Moreover, from $(i_k-1)\gamma = (i_k-1)\alpha = i_k\alpha - 1=c-1$ 
(notice that, if $(i_{k}-1)\alpha=i_k\alpha+1=(i_{k}+1)\alpha$ then, as $1\alpha=1$, it would exist 
$t\in\{1,\ldots,k-1\}$ such that $i_{t}\alpha>c$, which contradicts the definition of $i_{k}$) 
and $i_k\gamma = 2c - i_k\alpha =c$,
it follows that $|(i_k-1)\gamma - i_k\gamma| = 1$.
On the other hand, since we must have $(i_\ell+1)\alpha = i_\ell\alpha + 1$, 
from $i_\ell\gamma = 2c - i_\ell\alpha =2c-d$ and $(i_\ell+1)\gamma =(i_\ell+1)\alpha +2c-2d = i_\ell\alpha+1 + 2c-2d =2c-d+1$,  
it follows that $|i_\ell\gamma - (i_\ell+1)\gamma| = 1$.

Next, we show that $\Inv(\gamma) = \Inv(\alpha)\setminus \{i_{k},i_{\ell}\}$. 

Clearly,
if $x\in\{1,\ldots, i_{k}-2\} \cup \{i_{k}+1,\ldots,i_{\ell}-1\}\cup \{i_{\ell}+2,\ldots,n\}$ then $x\in \Inv(\gamma)$ if and only if $x\in \Inv(\alpha)$. 
Also
$$
\begin{array}{rcl}
    i_{k}-1\in \Inv(\gamma) & \Leftrightarrow& (i_{k}-2)\gamma=i_{k}\gamma \\
     & \Leftrightarrow& (i_{k}-2)\alpha = 2c - i_{k}\alpha \\
     & \Leftrightarrow& (i_{k}-2)\alpha = c = i_{k}\alpha \\
     & \Leftrightarrow& i_{k}-1\in \Inv(\alpha).
  \end{array}
$$
On the other hand, since $i_k\alpha = c$ and $(i_k-1)\alpha = (i_k+1)\alpha = c-1$, we have
$$
(i_k-1)\gamma = (i_k-1)\alpha = c-1,
$$
$$
i_k\gamma = 2c - i_k\alpha = c
$$
and
$$
(i_k+1)\gamma = 2c -(i_k+1)\alpha = 2c - (c-1) = c+1,
$$
whence $i_k \notin \Inv(\gamma)$.
Moreover,
$$
\begin{array}{rcl}
    i_{\ell}+1\in \Inv(\gamma) & \Leftrightarrow& i_{\ell}\gamma=(i_{\ell}+2)\gamma \\
    & \Leftrightarrow& 2c - i_{\ell}\alpha = (i_{\ell}+2)\alpha +2c-2d \\
    & \Leftrightarrow& 2c-d = (i_{\ell}+2)\alpha+2c-2d \\
    & \Leftrightarrow& (i_{\ell}+2)\alpha=d=i_{\ell}\alpha \\
    & \Leftrightarrow& i_{\ell}+1\in \Inv(\alpha).
  \end{array}
 $$
On the other hand, since $i_\ell\alpha = d$ and $(i_\ell-1)\alpha = (i_\ell+1)\alpha = d+1$, we have
$$
(i_\ell-1)\gamma = 2c-(i_\ell-1)\alpha = 2c -(d+1) = 2c-d-1,
$$
$$
i_\ell\gamma = 2c - i_\ell\alpha = 2c-d
$$
and
$$
(i_\ell+1)\gamma = (i_\ell+1)\alpha +2c-2d = d+1 + 2c-2d =2c-d+1,
$$
whence $i_\ell \notin \Inv(\gamma)$.

Therefore, $\inv(\gamma) = r-1$ and so, by induction, we have $\gamma \in \langle A' \rangle$.

Finally, we prove that $\Inv(\alpha) = \Inv(\gamma\beta_{c-d,d-1})$.

Let $i\in \Inv(\alpha)\cap\{1,\dots,i_{k}-2\}$. Then
$$
\begin{array}{rcl}
    (i-1)\gamma\beta_{c-d,d-1}=(i+1)\gamma\beta_{c-d,d-1} & \Leftrightarrow& (i-1)\alpha\beta_{c-d,d-1}=(i+1)\alpha\beta_{c-d,d-1} \\
    & \Leftrightarrow& (i-1)\alpha=(i+1)\alpha
  \end{array}
$$
(by the definition of $\beta_{c-d,d-1}$ restricted to $\{1,\ldots,c\}$).

Let $i\in \Inv(\alpha)\cap\{i_{k}+1,\dots,i_{\ell}-1\}$. Then
$$
\begin{array}{rcl}
    (i-1)\gamma\beta_{c-d,d-1}=(i+1)\gamma\beta_{c-d,d-1} & \Leftrightarrow& (2c-(i-1)\alpha)\beta_{c-d,d-1}=(2c-(i+1)\alpha)\beta_{c-d,d-1} \\
    & \Leftrightarrow& (c+(c-(i-1)\alpha))\beta_{c-d,d-1}=(c+(c-(i+1)\alpha))\beta_{c-d,d-1} \\
    & \Leftrightarrow& (i-1)\alpha=(i+1)\alpha
  \end{array}
$$
(by the definition of $\beta_{c-d,d-1}$ restricted to $\{c,\ldots,2c-d\}$).

Let $i\in \Inv(\alpha)\cap\{i_{\ell}+2,\dots,n\}$. Then 
$$
\begin{array}{ll}
  &   (i-1)\gamma\beta_{c-d,d-1}=(i+1)\gamma\beta_{c-d,d-1}\\
  \Leftrightarrow& ((i-1)\alpha+2c-2d)\beta_{c-d,d-1}=((i+1)\alpha+2c-2d)\beta_{c-d,d-1} \\
  \Leftrightarrow& (2c-d+((i-1)\alpha-d))\beta_{c-d,d-1}=(2c-d+((i+1)\alpha-d))\beta_{c-d,d-1} \\
  \Leftrightarrow& (i-1)\alpha=(i+1)\alpha
  \end{array}
 $$
(by the definition of $\beta_{c-d,d-1}$ restricted to $\{2c-d,\ldots,n\}$).

Moreover, if $i_{k}-1\in \Inv(\alpha)$ then $(i_{k}-2)\alpha=i_{k}\alpha=c$. Thus,  
$
(i_{k}-2)\gamma\beta_{c-d,d-1}=(i_{k}-2)\alpha\beta_{c-d,d-1}=c\beta_{c-d,d-1}=(2c-i_{k}\alpha)\beta_{c-d,d-1}=i_{k}\gamma\beta_{c-d,d-1} 
$
and so $i_{k}-1\in \Inv(\gamma\beta_{c-d,d-1})$.
Conversely, if $i_{k}-1\notin \Inv(\alpha)$ then $(i_{k}-2)\alpha\ne i_{k}\alpha=c$ and, as $(i_{k}-1)\alpha=i_{k}\alpha-1=c-1$, we have
$(i_{k}-2)\alpha=c-2$. Hence 
$
(i_{k}-2)\gamma\beta_{c-d,d-1}=(i_{k}-2)\alpha\beta_{c-d,d-1}=(c-2)\beta_{c-d,d-1}=c-2\ne c=
c\beta_{c-d,d-1}=(2c-c)\beta_{c-d,d-1}=(2c-i_k\alpha)\beta_{c-d,d-1}=i_{k}\gamma\beta_{c-d,d-1}
$
and so $i_{k}-1\notin \Inv(\gamma\beta_{c-d,d-1})$. 

Analogously, if $i_{\ell}+1\in \Inv(\alpha)$ then $i_{\ell}\alpha=(i_{\ell}+2)\alpha=d$, whence 
$
i_{\ell}\gamma\beta_{c-d,d-1}=(2c-i_{\ell}\alpha)\beta_{c-d,d-1}=(2c-d)\beta_{c-d,d-1}= d = ((i_{\ell}+2)\alpha+2c-2d)\beta_{c-d,d-1}=(i_{\ell}+2)\gamma\beta_{c-d,d-1}
$
and so $i_{\ell}+1\in \Inv(\gamma\beta_{c-d,d-1})$.
Conversely, if $i_{\ell}+1\notin \Inv(\alpha)$ then $d=i_{\ell}\alpha\ne(i_{\ell}+2)\alpha$. As $(i_{\ell}+1)\alpha=i_{\ell}\alpha+1=d+1$, then
$(i_{\ell}+2)\alpha=d+2$, whence 
$
i_{\ell}\gamma\beta_{c-d,d-1}=d\ne d+2=(2c-d+2)\beta_{c-d,d-1}=((i_{\ell}+2)\alpha+2c-2d)\beta_{c-d,d-1}=(i_{\ell}+2)\gamma\beta_{c-d,d-1}
$
and so $i_{\ell}+1\notin \Inv(\gamma\beta_{c-d,d-1})$.

It remains to show that $i_{k},i_{\ell}\in \Inv(\gamma\beta_{c-d,d-1})$. 
As $c, 2c-d \in \Inv(\beta_{c-d,d-1})$, $(i_{k}-1)\alpha = c-1 = (i_{k}+1)\alpha$ and 
$(i_{\ell}-1)\alpha = d+1 = (i_{\ell}+1)\alpha$, we have
$
(i_{k}-1)\gamma\beta_{c-d,d-1}=(i_{k}-1)\alpha\beta_{c-d,d-1}=(c-1)\beta_{c-d,d-1}= c-1 
=(c+1)\beta_{c-d,d-1}=(2c-(i_{k}+1)\alpha)\beta_{c-d,d-1}=(i_{k}+1)\gamma\beta_{c-d,d-1}, 
$
as well as
$
(i_{\ell}-1)\gamma\beta_{c-d,d-1}=(2c-(i_{\ell}-1)\alpha)\beta_{c-d,d-1}=(2c-d-1)\beta_{c-d,d-1}=d+1
=(2c-d+1)\beta_{c-d,d-1} = ((i_{\ell}+1)\alpha+2c-2d)\beta_{c-d,d-1}=(i_{\ell}+1)\gamma\beta_{c-d,d-1}.
$

Thus, we showed that $\Inv(\alpha)=\Inv(\gamma\beta_{c-d,d-1})$. Since $\gamma \in \langle A' \rangle$, then
$\gamma\beta_{c-d,d-1} \in \langle A' \rangle$ and so, by Lemma \ref{rem},
we have $\alpha \in \langle A' \rangle$, as required. 
\end{proof}

\medskip

Now, let us consider 
$$
A=\{\tau\}\cup\left\{\alpha_i\mid i=1,\ldots,\lfloor\frac{n-1}{2}\rfloor\right\}\cup\left\{\beta_{j,i}\mid j=1,\ldots,\lfloor\frac{n-3}{3}\rfloor,\, i=1,\ldots,\lfloor\frac{n-3j-1}{2}\rfloor\right\}.
$$

The next two lemmas together with the previous one show that $A$ is a generating set of $\End P_n$.  

\begin{lemma}\label{le1}
$\{\alpha_i \mid i = 1,2,\ldots,n-2\} \subseteq \langle A \rangle$.
\end{lemma}
\begin{proof}
If $i=1,\ldots,\lfloor\frac{n-1}{2}\rfloor$ then $\alpha_i \in A$ and so $\alpha_i\in\langle A \rangle$. 
On the other hand, we have 
$$
\alpha_{\lfloor\frac{n-1}{2}\rfloor + 1} = \tau\alpha_{\lfloor\frac{n-1}{2}\rfloor - 1},  \ldots, ~
\alpha_{n-3} = \tau\alpha_2, ~ \alpha_{n-2} = \tau\alpha_1,
$$
if $n$ is odd, and
$$
\alpha_{\lfloor\frac{n-1}{2}\rfloor + 1} = \tau\alpha_{\lfloor\frac{n-1}{2}\rfloor}, \ldots, ~
\alpha_{n-3} = \tau\alpha_2, ~ \alpha_{n-2} = \tau\alpha_1,
$$
if $n$ is even. Hence, we also have $\alpha_i\in\langle A \rangle$, for $i =\lfloor\frac{n-1}{2}\rfloor + 1,\ldots, n-2$, as required. 
\end{proof}

\begin{lemma}\label{le}
$\left\{\beta_{j,i}\mid j=1,\ldots,\lfloor\frac{n-3}{3}\rfloor,\, i=1,\ldots,n-3j-2\right\}\subseteq \langle A\rangle$.
\end{lemma}

\begin{proof}
Let $j=1,\ldots,\lfloor\frac{n-3}{3}\rfloor$. 
If $i=1,\ldots,\lfloor\frac{n-3j-1}{2}\rfloor$ then $\beta_{j,i} \in A$ and so $\beta_{j,i} \in \langle A\rangle$. 
On the other hand, it is a routine matter to verify that
$$
\ker(\beta_{j,\lfloor\frac{n-3j-1}{2}\rfloor + 1}) = \ker(\tau\beta_{j,\lfloor\frac{n-3j-1}{2}\rfloor}), ~
\ker(\beta_{j,\lfloor\frac{n-3j-1}{2}\rfloor + 2}) = \ker(\tau\beta_{j,\lfloor\frac{n-3j-1}{2}\rfloor - 1}), \ldots,$$
$$
\ker(\beta_{j,n-3j-3}) = \ker(\tau\beta_{j,2}), ~ \ker(\beta_{j,n-3j-2}) = \ker(\tau\beta_{j,1}),
$$
if $n-3j-2$ is even, and
$$
\ker(\beta_{j,\lfloor\frac{n-3j-1}{2}\rfloor + 1}) = \ker(\tau\beta_{j,\lfloor\frac{n-3j-1}{2}\rfloor - 1}), ~ 
\ker(\beta_{j,\lfloor\frac{n-3j-1}{2}\rfloor + 2}) = \ker(\tau\beta_{j,\lfloor\frac{n-3j-1}{2}\rfloor - 2}), \ldots,
$$
$$
\ker(\beta_{j,n-3j-3}) = \ker(\tau\beta_{j,2}), ~ \ker(\beta_{j,n-3j-2}) = \ker(\tau\beta_{j,1}), 
$$
if $n-3j-2$ is odd. Thus, in view of Lemmas \ref{le2} and \ref{le1}, for $i =\lfloor\frac{n-3j-1}{2}\rfloor + 1,\ldots, n-3j-2$, we conclude that also   
$\beta_{j,i} \in \langle A\rangle$,  as required. 
\end{proof}

So, as an immediate consequence of Lemmas \ref{ger}, \ref{le1} and \ref{le}, we obtain: 

\begin{proposition}\label{gend}
The set $A$ generates $\End P_n$. 
Moreover, $|A|=1+\lfloor\frac{n-1}{2}\rfloor+\sum_{j=1}^{\lfloor\frac{n-3}{3}\rfloor}\lfloor\frac{n-3j-1}{2}\rfloor$. 
\end{proposition}



\medskip 

Let 
$$
\gamma_i=\left(
\begin{array}{ccccccc}
1&\cdots&i&i+1&i+2&\cdots&n\\
1&\cdots&i&i&i+1&\cdots&n-1
\end{array}
\right),
$$
for $i=1,\ldots,\lfloor\frac{n}{2}\rfloor$. 
Let 
$$
B=A\cup\left\{\gamma_i\mid i=1,\ldots,\lfloor\frac{n}{2}\rfloor\right\}.
$$

Let $\alpha\in\wEnd P_n$. We say that $i\in\{1,\ldots,n-1\}$ is a \textit{repetition} of $\alpha$ if $(i)\alpha=(i+1)\alpha$.  
Denote by 
$
\rep(\alpha)
$
the number of repetitions of $\alpha$. 
This notion will be used in our next result. 
Observe that, clearly, $\alpha\in\End P_n$ if and only if $\rep(\alpha)=0$. 

\begin{proposition} \label{gwend} 
The set $B$ generates $\wEnd P_n$.
\end{proposition}
\begin{proof}
First, for $i=\lfloor\frac{n}{2}\rfloor+1,\ldots,n-1$, consider also 
$
\gamma_i=\left(
\begin{array}{ccccccc}
1&\cdots&i&i+1&i+2&\cdots&n\\
1&\cdots&i&i&i+1&\cdots&n-1
\end{array}
\right). 
$
Then, it is easy to check that $\gamma_{n-i}=\tau\gamma_i\tau\gamma_1$, for all $i=1,\ldots,\lfloor\frac{n}{2}\rfloor$. 
Hence $\gamma_1, \ldots, \gamma_{n-1}$ belong to the monoid generated by $B$. 

Now, in order to show that any $\alpha\in\wEnd P_n$ belongs to the monoid generated by $B$, we proceed by induction on $\rep(\alpha)$. 

Let $\alpha\in\wEnd P_n$ be such that $\rep(\alpha)=0$. Then, as observed above, $\alpha\in\End P_n$ and so, by Proposition \ref{gend}, we may conclude that $\alpha$ belongs to the monoid generated by $B$. 

Let $k\ge0$ and, by induction hypothesis, admit that 
$\alpha$ belongs to the monoid generated by $B$, for all $\alpha\in\wEnd P_n$ such that 
$\rep(\alpha)=k$ (notice that, for such $\alpha$ to exist, we must have $k\le n-1$). 

Let $\alpha\in\wEnd P_n$ be such that $\rep(\alpha)=k+1$ (by supposing that $k\le n-2$). 
Let $i\in\{1,\ldots,n-1\}$ be a repetition of $\alpha$ and take 
$
\beta=\left(
\begin{array}{ccccccc}
1&\cdots&i&i+1&\cdots&n-1&n\\
1\alpha&\cdots&i\alpha&(i+2)\alpha&\cdots&n\alpha& b
\end{array}
\right),  
$
where $b=n\alpha-1$, if $n\alpha\ge2$, or $b=n\alpha+1$, otherwise. 
It is clear that $\beta\in\wEnd P_n$ and $\rep(\beta)=\rep(\alpha)-1=k$, whence $\beta$ belongs to the monoid generated by $B$, by induction hypothesis. On the other hand, it is a routine matter to check that $\alpha=\gamma_i\beta$ and so we may conclude that 
$\alpha$ belongs to the monoid generated by $B$, as required. 
\end{proof} 

Observe that $|B|=n+\sum_{j=1}^{\lfloor\frac{n-3}{3}\rfloor}\lfloor\frac{n-3j-1}{2}\rfloor$.

\medskip 

In order to compute the ranks of $\End P_n$ and $\wEnd P_n$, we start by proving a series of lemmas involving the notion of inversion. 

\begin{lemma}\label{1w}
Let $\alpha,\beta\in\wEnd P_n$ be such that $\alpha\beta\in\End P_n$. 
Then, we have:
\begin{enumerate}
\item $\alpha\in\End P_n$;
\item   
$
\Inv(\alpha)\subseteq\Inv(\alpha\beta); 
$ 
\item   
$
\{i\in\{2,\ldots,n-1\}\mid i\alpha\in\Inv(\beta)\}\subseteq\Inv(\alpha\beta).
$ 
\end{enumerate} 
\end{lemma}
\begin{proof}
1. If $\alpha\in\wEnd P_n\setminus\End P_n$ then $i\alpha=(i+1)\alpha$, for some $i\in\{1,\ldots,n-1\}$, and so $i\alpha\beta=(i+1)\alpha\beta$, 
whence $\alpha\beta\not\in\End P_n$, which is a contradiction. Thus, $\alpha\in\End P_n$. 

\smallskip 

2. Let $i\in\Inv(\alpha)$. Then $i\in\{2,\ldots,n-1\}$ and $(i-1)\alpha=(i+1)\alpha$. Hence $(i-1)\alpha\beta=(i+1)\alpha\beta$ and so 
$i\in\Inv(\alpha\beta)$. 

\smallskip 

3. Let $i\in\{2,\ldots,n-1\}$ be such that $i\alpha\in\Inv(\beta)$. 
Then $2\le i\alpha\le n-1$ and $(i\alpha-1)\beta=(i\alpha+1)\beta$. 

If $(i-1)\alpha=(i+1)\alpha$ then $(i-1)\alpha\beta=(i+1)\alpha\beta$ and so $i\in\Inv(\alpha\beta)$. 

Let us suppose that $(i-1)\alpha\ne(i+1)\alpha$. Then, either $(i-1)\alpha=i\alpha-1$ and $(i+1)\alpha=i\alpha+1$ or 
$(i-1)\alpha=i\alpha+1$ and $(i+1)\alpha=i\alpha-1$, from which follows that 
$(i-1)\alpha\beta=(i\alpha\mp1)\beta=(i\alpha\pm1)\beta=(i+1)\alpha\beta$ and so $i\in\Inv(\alpha\beta)$, as required. 
\end{proof}

\begin{lemma}\label{1wpp}
Let $\alpha,\beta\in\wEnd P_n$. Let $i\in\Inv(\alpha\beta)$ be such that $i\not\in\Inv(\alpha)$. Then $i\alpha\in\Inv(\beta)$. 
\end{lemma}
\begin{proof}
As $i\in\Inv(\alpha\beta)$, we have $2\le i\le n-1$ and $(i-1)\alpha\beta=(i+1)\alpha\beta\ne i\alpha\beta$. 
In addition, as $2\le i\le n-1$ and $i\not\in\Inv(\alpha)$, we have $(i-1)\alpha\ne(i+1)\alpha$ or $(i-1)\alpha=(i+1)\alpha=i\alpha$. 
If $(i-1)\alpha=(i+1)\alpha=i\alpha$ then $(i-1)\alpha\beta=(i+1)\alpha\beta=i\alpha\beta$, which is a contradiction. 
Hence $(i-1)\alpha\ne(i+1)\alpha$. Moreover, if $(i-1)\alpha=i\alpha$ or $(i+1)\alpha=i\alpha$ then $(i-1)\alpha\beta=i\alpha\beta$ 
or $(i+1)\alpha\beta=i\alpha\beta$, which also is a contradiction. Thus, either $(i-1)\alpha=i\alpha-1$ and $(i+1)\alpha=i\alpha+1$ or 
$(i-1)\alpha=i\alpha+1$ and $(i+1)\alpha=i\alpha-1$ (and, in both cases, we must have $2\le i\alpha\le n-1$), whence 
$(i\alpha\mp1)\beta=(i-1)\alpha\beta=(i+1)\alpha\beta=(i\alpha\pm1)\beta$ and so $(i\alpha-1)\beta=(i\alpha+1)\beta$. 
Since $(i+1)\alpha\beta\ne i\alpha\beta$, we have $(i\alpha-1)\beta=(i\alpha+1)\beta\ne i\alpha\beta$, 
i.e. $i\alpha\in\Inv(\beta)$, as required. 
\end{proof}

\begin{lemma}\label{1wppp}
Let $\alpha\in\wEnd P_n$ and $i\in\{2,\ldots,n-1\}$. 
Then $i\in\Inv(\alpha)$ if and only if $n-i+1\in\Inv(\tau\alpha)$. 
\end{lemma}
\begin{proof}
First, notice that $2\le i\le n-1$ if and only if $2\le n-i+1\le n-1$. Then 
$$
\begin{array}{rcl}
i\in\Inv(\alpha) & \Leftrightarrow & (i-1)\alpha=(i+1)\alpha\ne i\alpha \\
&   \Leftrightarrow & (n-i+2)\tau\alpha=(n-i)\tau\alpha\ne (n-i+1)\tau\alpha \\
&   \Leftrightarrow & ((n-i+1)+1)\tau\alpha=((n-i+1)-1)\tau\alpha\ne (n-i+1)\tau\alpha \\
&   \Leftrightarrow &n-i+1\in\Inv(\tau\alpha),
\end{array}
$$
as required. 
\end{proof}

The next lemma is clear. 

\begin{lemma}\label{2w}
Let $C$ be a generating set of $\End P_n$ or of $\wEnd P_n$. Then $\tau\in C$. 
\end{lemma}

Notice that $\Inv(\tau)=\emptyset$. Moreover, for $\alpha\in\End P_n$, 
we have $\Inv(\alpha)=\emptyset$ if and only if $\alpha=1$ or $\alpha=\tau$. 

\begin{lemma}\label{3w}
Let $C$ be a generating set of $\End P_n$ or of $\wEnd P_n$. Then $C$ possesses at least $\lfloor\frac{n-1}{2}\rfloor$ 
distinct transformations $\alpha$ with $\inv(\alpha)=1$. 
\end{lemma}
\begin{proof}
In order to obtain a contradiction, let us assume that $C$ contains less than $\lfloor\frac{n-1}{2}\rfloor$ 
distinct transformations $\alpha$ with $\inv(\alpha)=1$. Then, there exists $i\in\{2,\ldots,\lfloor\frac{n+1}{2}\rfloor\}$ such that 
$\{i,n-i+1\}\cap\Inv(\alpha)=\emptyset$, for all $\alpha\in C$. 
As $1\le i-1\le \lfloor\frac{n-1}{2}\rfloor$, we may consider the transformation $\alpha_{i-1}\in A$. 

Let $\xi_1,\ldots,\xi_k\in C\setminus\{1\}$ be such that $\alpha_{i-1}=\xi_1\cdots\xi_k$ and $\{\xi_j,\xi_{j+1}\}\ne\{\tau\}$, for $j=1,\ldots,k-1$. 
Notice that $\Inv(\alpha_{i-1})=\{i\}$, whence $\alpha_{i-1}\not\in C$ and so $k\ge2$. 
Moreover, as $\alpha_{i-1}\in\End P_n$, 
by Lemma \ref{1w}, we have $\xi_1\in\End P_n$ and 
$
\Inv(\xi_1)\subseteq\Inv(\xi_1\cdots\xi_k)=\Inv(\alpha_{i-1})=\{i\}. 
$
Then $\Inv(\xi_1)=\emptyset$, since $\xi_1\in C$, and so $\xi_1=\tau$ (since $\xi_1\in\End P_n$ and $\xi_1\ne1$). 

Applying Lemma \ref{1w} again, we obtain 
$\Inv(\tau\xi_2)=\Inv(\xi_1\xi_2)\subseteq\Inv(\xi_1\cdots\xi_k)=\Inv(\alpha_{i-1})=\{i\}$ and 
$\tau\xi_2\in\End P_n$. 
Hence, $\Inv(\tau\xi_2)=\emptyset$ or $\Inv(\tau\xi_2)=\{i\}$. 

Suppose that $\Inv(\tau\xi_2)=\emptyset$.  Then $\tau\xi_2=1$ or 
$\tau\xi_2=\tau$, and so $\xi_2=\tau$ or $\xi_2=1$, which is not possible since $\{\xi_1,\xi_2\}\ne\{\tau\}$ and $\xi_2\ne1$. 
Thus, we must have $\Inv(\tau\xi_2)=\{i\}$ and so, by Lemma \ref{1wppp}, it follows that $\Inv(\xi_2)=\{n-i+1\}$, 
which is a contradiction, since $\xi_2\in C$. 

Therefore, $C$ must contain at least $\lfloor\frac{n-1}{2}\rfloor$ 
distinct transformations $\alpha$ with $\inv(\alpha)=1$, as required. 
\end{proof}

\begin{lemma}\label{uneq}
Let $j\in\{1,\ldots,\lfloor\frac{n-3}{3}\rfloor\}$ and $i\in\{1,\ldots,\lfloor\frac{n-3j-1}{2}\rfloor\}$. Then 
$$
2\le i+1\le n-2, \quad 4\le i+2j+1\le n-2\quad\text{and}\quad 5\le i+3j+1\le n-1. 
$$
\end{lemma}
\begin{proof}
We have $2\le i+1\le \frac{n-3j-1}{2}+1\le \frac{n-3-1}{2}+1=\frac{n}{2}-1< n-1$. 
On the other hand, $4\le i+2j+1\le i+j+j+1\le\frac{n-3j-1}{2} + \frac{n-3}{3} +j+1 
=\frac{5n-3j-9}{6}+1\le \frac{5n-3-9}{6}+1=\frac{5n}{6}-1<n-1$.  
Finally, $5\le i+3j+1\le  \frac{n-3j-1}{2}+3j+1=\frac{n+3j+1}{2}\le \frac{n+3\cdot\frac{n-3}{3}+1}{2}=n-1$, as required. 
\end{proof}

\begin{lemma}\label{4w}
Let $C$ be a generating set of $\End P_n$ or of $\wEnd P_n$. Then $C$ possesses at least 
$\sum_{j=1}^{\lfloor\frac{n-3}{3}\rfloor}\lfloor\frac{n-3j-1}{2}\rfloor$
distinct transformations $\alpha$ with $\inv(\alpha)=2$. 
\end{lemma} 
\begin{proof}
Let $j\in\{1,\ldots,\lfloor\frac{n-3}{3}\rfloor\}$ and $i\in\{1,\ldots,\lfloor\frac{n-3j-1}{2}\rfloor\}$. 
In order to obtain a contradiction, let us assume that $\Inv(\alpha)\ne\{i+j+1,i+2j+1\}$ and  $\Inv(\alpha)\ne\{n-(i+j+1)+1,n-(i+2j+1)+1\}$, 
for all $\alpha\in C$. 

Let us consider the transformation $\beta_{j,i}\in\End P_n$. 
Observe that $\Inv(\beta_{j,i})=\{i+j+1,i+2j+1\}$, whence $\beta_{j,i}\not\in C$. 
Let $\xi_1,\ldots,\xi_k\in C\setminus\{1\}$ be such that $\beta_{j,i}=\xi_1\cdots\xi_k$ and $\{\xi_\ell,\xi_{\ell+1}\}\ne\{\tau\}$, 
for $\ell=1,\ldots,k-1$. 
Notice that $k\ge2$, since $\beta_{j,i}\not\in C$. 
As $\beta_{j,i}\in\End P_n$, 
by Lemma \ref{1w}, we have $\xi_1\in\End P_n$ and 
$
\Inv(\xi_1)\subseteq\Inv(\xi_1\cdots\xi_k)=\Inv(\beta_{j,i})=\{i+j+1,i+2j+1\}. 
$ 
Since $\xi_1\in C$, then $\inv(\xi_1)=0$ or $\inv(\xi_1)=1$. 

If $\inv(\xi_1)=1$ then $\Inv(\xi_1)\in\{\{i+j+1\},\{i+2j+1\}\}$. 

On the other hand, suppose that $\inv(\xi_1)=0$. As $\xi_1\in\End P_n$ (and $\xi_1\ne1$), then we must have $\xi_1=\tau$. 
By Lemma \ref{1w}, we get  
$\Inv(\tau\xi_2)=\Inv(\xi_1\xi_2)\subseteq\Inv(\xi_1\cdots\xi_k)=\Inv(\beta_{j,i})=\{i+j+1,i+2j+1\}$ and 
$\tau\xi_2\in\End P_n$. 
It follows that $\xi_2\in\End P_n$ and, 
by Lemma \ref{1wppp}, that $\Inv(\xi_2)\subseteq \{n-(i+j+1)+1,n-(i+2j+1)+1\}$. 
As $\xi_2\in C$, we obtain $\inv(\xi_2)=0$ or $\inv(\xi_2)=1$. 
If $\inv(\xi_2)=0$ then $\xi_2=\tau$ (since $\xi_2\ne1$ and $\xi_2\in\End P_n$) and so $\{\xi_1,\xi_2\}=\{\tau\}$, 
which is a contradiction. Thus $\Inv(\xi_2)\in\{\{n-(i+j+1)+1\},\{n-(i+2j+1)+1\}\}$. 
Also, notice that, in this case, $k\ge3$ (since $k=2$ would imply $\xi_2=\tau\beta_{j,i}$ and so $\inv(\xi_2)=2$, which is a contradiction). 

Therefore, we have four cases to consider. 

\smallskip 

{\sc case}  1. $\Inv(\xi_1)=\{i+j+1\}$. Then, as $\xi_1\in\End P_n$, we must have $(i+2j+1)\xi_1=(i+1)\xi_1$. 
On the other hand, since $i+2j+1\in\Inv(\beta_{j,i})=\Inv(\xi_1(\xi_2\cdots\xi_k))$ and $i+2j+1\not\in\Inv(\xi_1)$, by Lemma \ref{1wpp}, we obtain 
$(i+2j+1)\xi_1\in\Inv(\xi_2\cdots\xi_k)$.  Thus $(i+1)\xi_1\in \Inv(\xi_2\cdots\xi_k)$.  
Now,  as $2\le i+1\le n-2$ (by Lemma \ref{uneq}), it follows by Lemma \ref{1w} that 
$i+1\in \Inv(\xi_1(\xi_2\cdots\xi_k))=\Inv(\beta_{j,i})$, which is a contradiction. 

\smallskip 

{\sc case}  2. $\Inv(\xi_1)=\{i+2j+1\}$. 
Notice that $5\le i+3j+1\le n-1$, by Lemma \ref{uneq}. 
As $\xi_1\in\End P_n$, in this case, we have $(i+3j+1)\xi_1=(i+j+1)\xi_1$. 
Since $i+j+1\in\Inv(\beta_{j,i})=\Inv(\xi_1(\xi_2\cdots\xi_k))$ and $i+j+1\not\in\Inv(\xi_1)$, by Lemma \ref{1wpp}, we obtain 
$(i+j+1)\xi_1\in\Inv(\xi_2\cdots\xi_k)$, i.e. $(i+3j+1)\xi_1\in \Inv(\xi_2\cdots\xi_k)$.  
Hence,  by Lemma \ref{1w}, we get $i+3j+1\in \Inv(\xi_1(\xi_2\cdots\xi_k))=\Inv(\beta_{j,i})$, which is a contradiction. 

\smallskip 

Before considering the next case, we observe that $\Inv(\tau\beta_{j,i})=\{n-(i+j+1)+1,n-(i+2j+1)+1\}$, by Lemma \ref{1wppp}. 

\smallskip 

{\sc case}  3. $\xi_1=\tau$ and $\Inv(\xi_2)=\{n-(i+j+1)+1\}$. 
Since $\xi_2\in\End P_n$, we deduce that $(n-(i+1)+1)\xi_2=(n-(i+2j+1)+1)\xi_2$. 
Moreover, as $n-(i+2j+1)+1\in\Inv(\tau\beta_{j,i})$,  $n-(i+2j+1)+1\not\in\Inv(\xi_2)$ and $\tau\beta_{j,i}=\xi_2\xi_3\cdots\xi_k$ (notice that, in this case, $k\ge3$), by Lemma \ref{1wpp}, we have $(n-(i+2j+1)+1)\xi_2\in\Inv(\xi_3\cdots\xi_k)$. 
Thus, $(n-(i+1)+1)\xi_2\in\Inv(\xi_3\cdots\xi_k)$.  From $2\le i+1\le n-2$ (by Lemma \ref{uneq}), we obtain $3\le n-(i+1)+1\le n-1$ and so, 
by Lemma \ref{1w}, it follows that $n-(i+1)+1\in\Inv(\xi_2(\xi_3\cdots\xi_k))=\Inv(\tau\beta_{j,i})$, which is a contradiction. 

\smallskip 

{\sc case}  4. $\xi_1=\tau$ and $\Inv(\xi_2)=\{n-(i+2j+1)+1\}$. Once again  
since $\xi_2\in\End P_n$, we conclude that $(n-(i+j+1)+1)\xi_2=(n-(i+3j+1)+1)\xi_2$. 
On the other hand, as $n-(i+j+1)+1\in\Inv(\tau\beta_{j,i})$,  $n-(i+j+1)+1\not\in\Inv(\xi_2)$ and $\tau\beta_{j,i}=\xi_2\xi_3\cdots\xi_k$ ($k\ge3$, also in this case), by Lemma \ref{1wpp}, we have $(n-(i+j+1)+1)\xi_2\in\Inv(\xi_3\cdots\xi_k)$ and so 
$(n-(i+3j+1)+1)\xi_2\in\Inv(\xi_3\cdots\xi_k)$.  By Lemma \ref{uneq}, we have $5\le i+3j+1\le n-1$, whence $2\le n-(i+3j+1)+1\le n-4$. Hence, 
by Lemma \ref{1w}, we obtain $n-(i+3j+1)+1\in\Inv(\xi_2(\xi_3\cdots\xi_k))=\Inv(\tau\beta_{j,i})$, which is a contradiction. 

\smallskip 

Since we obtained a contradiction in all possible cases, it follows that $\Inv(\alpha)=\{i+j+1,i+2j+1\}$ or 
$\Inv(\alpha)=\{n-(i+j+1)+1,n-(i+2j+1)+1\}$, 
for some $\alpha\in C$. Therefore $C$ has at least 
$\sum_{j=1}^{\lfloor\frac{n-3}{3}\rfloor}\lfloor\frac{n-3j-1}{2}\rfloor$
distinct transformations $\alpha$ with $\inv(\alpha)=2$, as required. 
\end{proof}

Now, as an immediate consequence of Proposition \ref{gend} and Lemmas \ref{2w}, \ref{3w} and \ref{4w}, we have: 

\begin{theorem}
The rank of $\End P_n$ is equal to $1+\lfloor\frac{n-1}{2}\rfloor+\sum_{j=1}^{\lfloor\frac{n-3}{3}\rfloor}\lfloor\frac{n-3j-1}{2}\rfloor$. 
\end{theorem}

To calculate the rank of $\wEnd P_n$, we still need the following lemma. 

\begin{lemma}\label{5w}
Let $C$ be a generating set of $\wEnd P_n$. Then $C$ possesses at least 
$\lfloor\frac{n}{2}\rfloor$ distinct transformations $\alpha\in\wEnd P_n\setminus\End P_n$ with $\inv(\alpha)=0$. 
\end{lemma}
\begin{proof}
Let $i\in\{1,\ldots,\lfloor\frac{n}{2}\rfloor\}$. 
Let $\xi_1,\ldots,\xi_k\in C\setminus\{1\}$ be such that $\gamma_i=\xi_1\cdots\xi_k$ and $\{\xi_j,\xi_{j+1}\}\ne\{\tau\}$, for $j=1,\ldots,k-1$. 
Then 
$$
\ker(\xi_1)\subseteq\ker(\gamma_i)=\text{id}_n \cup\{(i,i+1),(i+1,i)\},
$$ 
whence $\xi_1$ is a permutation of $\{1,\ldots,n\}$ or $\ker(\xi_1)=\ker(\gamma_i)$ and so, as $\xi_1\ne1$, 
$\xi_1=\tau$ or $\ker(\xi_1)=\ker(\gamma_i)$. 

Suppose that $\xi_1=\tau$. Then $k\ge2$ and $\tau\gamma_i=\xi_2\cdots\xi_k$. Hence 
$$
\ker(\xi_2)\subseteq\ker(\tau\gamma_i)=\text{id}_n \cup\{(n-i,n-i+1),(n-i+1,n-i)\}. 
$$
Since $\{\xi_1,\xi_2\}\ne\{\tau\}$, then $\xi_2\ne\tau$ and so 
$$
\ker(\xi_2)=\ker(\tau\gamma_i)=\text{id}_n \cup\{(n-i,n-i+1),(n-i+1,n-i)\}. 
$$

Therefore,  $C$ possesses at least $\lfloor\frac{n}{2}\rfloor$ (distinct) transformations $\alpha$ such that 
$$
\ker(\alpha)=\text{id}_n \cup\{(i,i+1),(i+1,i)\},
$$
for some $i\in\{1,\ldots,n-1\}$. Clearly, $\alpha\in\wEnd P_n\setminus\End P_n$ and $\inv(\alpha)=0$, 
for all transformations $\alpha$ with this type of kernel. 
This completes the proof of the lemma. 
\end{proof}

Recall that the \textit{relative rank} of a semigroup (or a monoid) $S$ \textit{modulo} $X$, where $X$ is a subset of $S$, 
is the minimum cardinality of a subset $Y$ of $S$ such that $S$ is generated by $X\cup Y$. 

\smallskip 

Finally, Proposition \ref{gwend} together with Lemmas \ref{2w}, \ref{3w}, \ref{4w} and \ref{5w}, allow us to conclude: 

\begin{theorem}
The rank of $\wEnd P_n$ is equal to $n+\sum_{j=1}^{\lfloor\frac{n-3}{3}\rfloor}\lfloor\frac{n-3j-1}{2}\rfloor$. 
Moreover, the relative rank of $\wEnd P_n$ modulo $\End P_n$ is equal to $\lfloor\frac{n}{2}\rfloor$. 
\end{theorem}

\section*{Acknowledgement} 

This work was produced, in part, during the visit of the first and third authors to CMA, FCT NOVA, Lisbon, in March 2017.  
The first author was supported by CMA through a visiting researcher fellowship. 



{\small \sf  
\noindent{\sc Ilinka Dimitrova}, 
Faculty of Mathematics and Natural Science, 
South-West University "Neofit Rilski", 
2700 Blagoevgrad, 
Bulgaria;  
email: ilinka\_dimitrova@swu.bg.

\medskip 

\noindent{\sc V\'\i tor H. Fernandes}, 
CMA, Departamento de Matem\'atica, 
Faculdade de Ci\^encias e Tecnologia, 
Universidade NOVA de Lisboa, 
Monte da Caparica, 
2829-516 Caparica, 
Portugal; 
e-mail: vhf@fct.unl.pt. 

\medskip

\noindent{\sc J\"{o}rg Koppitz},  
Institute of Mathematics and Informatics Bulgarian Academy of Sciences,
 1113 Sofia, Bulgaria; 
 e-mail address: koppitz@math.bas.bg.

\medskip 

\noindent{\sc Teresa M. Quinteiro}, 
Instituto Superior de Engenharia de Lisboa, 
1950-062 Lisboa, 
Portugal. 
Also: 
Centro de Matem\'atica e Aplica\c{c}\~oes, 
Faculdade de Ci\^encias e Tecnologia, 
Universidade Nova de Lisboa, 
Monte da Caparica, 
2829-516 Caparica, 
Portugal;
e-mail: tmelo@adm.isel.pt. 
} 

\end{document}